\documentclass[graybox]{svmult}

\usepackage{mathptmx}       
\usepackage{helvet}         
\usepackage{courier}        
\usepackage{type1cm}        
\usepackage{makeidx}         
\usepackage{graphicx}        
\usepackage{multicol}        
\usepackage[bottom]{footmisc}

\usepackage{amsfonts,amssymb,amsmath,amsthm}
\DeclareMathOperator*{\argmin}{arg\,min}
\usepackage{ dsfont }

\makeindex             

\usepackage[inoutnumbered,ruled,english]{algorithm2e}
\SetAlgorithmName{Algorithm}{Nom de l'algorithme}{Liste des algorithmes}
\SetKwSty{}         
\DontPrintSemicolon 
\SetInd{1.8em}{0em}  

\SetKwInput{Variables}{\textbf{VARIABLES}}
\SetKwInput{Entree}{\textbf{Input}}
\SetKwInput{Init}{\textbf{Initialisation}}
\SetKwInput{Sortie}{\textbf{Output}}
\SetKwInput{rep}{\textbf{Repeat}}
\SetKwInput{until}{\textbf{Until}}

\begin{document}

\title*{Research Report: Exact biconvex reformulation of the $\ell_2-\ell_0$ minimization problem}
\author{Arne Bechensteen, Laure Blanc-F\'{e}raud, and Gilles Aubert
}
\institute{Arne Bechensteen \at Universit\'{e} C\^{o}te d'Azur, CNRS, INRIA, Laboratoire I3S UMR 7271, 06903 Sophia Antipolis, France, \email{arne-henrik.bechensteen@inria.fr }
\and Laure Blanc-F\'{e}raud \at Universit\'{e} C\^{o}te d'Azur, CNRS, INRIA, Laboratoire I3S UMR 7271, 06903 Sophia Antipolis, France, \email{blancf@i3s.unice.fr}
\and
Gilles Aubert \at Universit\'{e} C\^{o}te d'Azur, UNS, Laboratoire J. A. Dieudonn\'{e} UMR 7351, 06100 Nice, France,$\,\,$ $\,\,$ $\,\,$   \email{gaubert@unice.fr}}
%
%
\maketitle
\abstract{We focus on the minimization of the least square loss function either under a $k$-sparse constraint or with a sparse penalty term. Based on recent results, we reformulate the $\ell_0$ pseudo-norm exactly as a convex minimization problem by introducing an auxiliary variable. We then propose an exact biconvex reformulation of the $\ell_2-\ell_0$ constrained and penalized problems. We give correspondence results between minimizers of the initial function and the reformulated ones.  The  reformulation is biconvex and the non-convexity is due to a penalty term. These two properties  are used to derive a minimization algorithm. We apply the algorithm to the problem of single-molecule localization microscopy and compare the results with the well-known Iterative Hard Thresholding algorithm. Visually and numerically the biconvex reformulations perform better.} 

\section{Introduction}
\label{sec:1}
Sparse optimization consists in finding a solution  with many zero components from an underdetermined problem. There are many problems where the solution has many zero components (e.g machine learning, variable selection, pulse deconvolution, etc).  The most common way to measure the sparsity of a solution is by using the counting function $\|\cdot\|_0$ which is, by abuse of terminology, referred to as the $\ell_0$-norm, and is defined as

$$\|x\|_0=\#\{x_i , i=1,\cdots N : x_i\neq 0\}.$$

In this paper, we are interested in linear problems where the observation $d\in \mathbb{R}^M$ can be described as the multiplication of the solution $x\in \mathbb{R}^N$  with a matrix $A\in \mathbb{R}^{M\times N}$ plus some noise $\eta$ which we assume to be additive white Gaussian and independent of the data. 
$$
Ax+\eta=d
$$
This problem is underdetermined when $M<N$.
In sparse optimization involving the square norm, there are  three different approaches to tackle the problem. We search for $\hat{x}$ solution of:  
\begin{align}
    &\hat{x} \in \argmin_x G_{\ell_0}(x):=\frac{1}{2}\|Ax-d\|^2_2+\lambda \|x\|_0 \tag{P} \label{eq:sparspenalty}\\
   &\hat{x} \in \argmin_{x}G_{k}(x):=\frac{1}{2}\|Ax-d\|^2_2\text{ s.t. } \|x\|_0\leq k \tag{C} \label{eq:sparseconst}\\
    &\hat{x} \in \argmin_{x} \|x\|_0  \text{ s.t. } \frac{1}{2}\|Ax-d\|^2_2\leq \varepsilon \label{eq:sparsnouse}
\end{align} 
The two cases (\ref{eq:sparseconst}) and (\ref{eq:sparsnouse})  are on constrained form. For problem (\ref{eq:sparseconst}), the user has a knowledge of the sparsity of the solution, which is in this case at maximum $k$. In the case of (\ref{eq:sparsnouse}), the user have prior knowledge about the amount of noise, $\eta$, the signal $d$ has been affected by. It is usually possible to estimate $\epsilon$ from the data and the statistics of $\eta$. 

In the case of problem (\ref{eq:sparspenalty}), which is also referred to as the penalized $\ell_0$ form, the user does not have any information on the sparsity of the solution, nor on the noise the signal $d$ has been affected by. Therefore, the user must choose the amplitude of $\lambda\in \mathbb{R}_+$ which serves as a trade-off parameter between the data term and the sparsity term. If $\lambda$ is large, the reconstruction of $x$ will be sparse, but the difference between $Ax$ and $d$ may be large. Conversely, if $\lambda$ is small, the error between $Ax$ and $d$ is small, but the reconstructed $x$ may not be sparse. 

The above problems are not continuous, nor convex and the problems are known to be NP-hard due to the combinatorial nature of the $\ell_0$ -norm. However, they have been greatly studied due to their countless applications such as  sparse reconstruction of signals, variable selection, and single-molecule localization microscopy to cite a few. There are two main approaches to solve the problems, which are greedy algorithms and relaxations. A new approach has been lately been introduced which is a mathematical program with equilibrium constraint.

\textbf{Greedy algorithms}
Greedy algorithms are often used in sparse optimization. The idea behind these algorithms is to start with a zero initialization and for each iteration add one component to the signal $x$ until the wished sparsity is obtained. One of the easiest and least costly greedy algorithms, the Matching Pursuit (MP) algorithm \cite{mallat_matching_1993}, adds the component that reduces the residual $R$ at each iteration $s$, which is defined as $R=d-Ax^s$. 

The Orthogonal Matching Pursuit (OMP), proposed in \cite{pati_orthogonal_1993}, is a more refined version of MP. The algorithm chooses each component in the same way as MP, but for each new component added, it calculates and update the value of all the previous components. This may lead to a better result, but the complexity and cost of the calculation are greater than MP. 

Greedy algorithms have been greatly studied and a lot of different versions of the above algorithms has been developed. More complex ones, as the algorithm Greedy sparse simplex \cite{beck_sparsity_2012} or Single Best Replacement (SBR) \cite{soussen2011bernoulli} have been introduced. In contrast to MP or OMP, the algorithms can add and also substract components.

\textbf{Relaxations}
The three formulations of the sparse optimization problem (\ref{eq:sparspenalty}, \ref{eq:sparseconst} and \ref{eq:sparsnouse}) are non-convex, due to the non-convexity of the  $\ell_0$-norm. A common alternative is to work with the convex  $\ell_1$-norm instead of non-convex $\ell_0$-norm. The $\ell_2-\ell_0$ problem becomes a $\ell_2-\ell_1$ problem. This is called a convex relaxation since the non-convex term is replaced by a convex term. However, only under certain assumptions, the original problem and the convex relaxed problem have the same solutions \cite{candes_robust_2006}. Furthermore, $\|x\|_0$ and $\|x\|_1$ are very different  when $x$ contains large values. Non-smooth, non-convex but continuous relaxations where primarily introduced to avoid the difference between $\|x\|_0$ and $\|x\|_1$ when $x$ contains large values.  These relaxations are still non-convex, and the convergence of the algorithms to a global minimum is not assured. There are many non-convex continuous relaxations, as the NonNegative Garrote \cite{breiman_better_1995}, the Log-Sum penalty \cite{candes_enhancing_2007} or Capped-$\ell_1$ \cite{peleg_bilinear_2008} to mention some. The continuous Exact $\ell_0$ penalty introduced in \cite{soubies2015continuous} proposes an exact relaxation for the problem (\ref{eq:sparspenalty}) and a unified view of these functions is given in \cite{soubies2017unified}.

\textbf{Mathematical program with equilibrium constraint}
A more recent method resolving a sparse optimization problem is to introduce auxiliary variables to simulate the nature of $\ell_0$-norm and add a constraint between the primary variable and the auxiliary.  Hence the problem becomes a mathematical program with equilibrium constraint, and among the approaches, we find Mixed Integer reformulations \cite{bourguignon2016exact}, Boolean relaxation \cite{pilanci2015sparse} and the article that inspired our work, \cite{yuan_sparsity_2016}. The method has been used to study the three formulations of the sparse optimization problem (see \cite{bi2014exact, lu2013sparse}, for example).

\textbf{Contribution}: 
The aim of this paper is to present and study a new method for optimizing the constrained (\ref{eq:sparseconst}) and penalized (\ref{eq:sparspenalty}) problem with an added positivity constraint. The added positivity constraint is important in many sparse optimization problems. We start in section 2 by introducing a reformulation of the $\ell_0$ -norm by a variational characterization. The norm is rewritten as a convex minimization problem by introducing an auxiliary variable, and we can reformulate (\ref{eq:sparseconst}) and (\ref{eq:sparspenalty}) as a mathematical program with equilibrium constraint (MPEC). The reformulation of the $\ell_0$ -norm was presented in \cite{yuan_sparsity_2016}, and our work is an extension of their work, as they only study the minimization of a data term which is  Lipschitz continuous with a sparsity constraint. In this paper, the data term is the square norm,on the error $Ax-d$ which is not Lipschitz continuous, and we study the minimization with a sparsity constraint (problem (\ref{eq:sparseconst})) and with a sparse penalty term (problem (\ref{eq:sparspenalty})).  Based on the MPEC formulation of the problem we define a Lagrangian cost function $G_\rho$. The  function $G_\rho : \mathbb{R}^N\times\mathbb{R}^N\rightarrow \mathbb{R}$ is biconvex. The main contribution of the paper is Theorem \ref{theo:vikC} for the constrained version of $G_\rho$ and Theorem \ref{theo:vikP} for the penalized version of $G_\rho$.  We show that minimizing the $G_\rho$ is equivalent, in the sense of minimizers, as to find a solution to the initial constrained or penalized problem. In section 3 we propose an algorithm to minimize the new objective function. This algorithm is easy to implement as it is based on already existing and well known algorithms.  In section 4 we test the algorithms on the problem of single-molecule localization microscopy (SMLM). This is a well-studied problem \cite{sage2015quantitative} where the goal is to localize the molecules with a high precision.   

\textbf{Notations}:
\begin{itemize}
\item $\|\cdot\|=\|\cdot\|_2$, the $\ell_2$-norm. If another norm is applied, this will be denoted with a subscript.
\item The function
$$
\|x\|_0=\#\{x_i , i=1,\cdots N : x_i\neq 0\}
$$
will be, by abuse of terminology, referred to as the  $\ell_0$-norm.
\item The observed signal  $d\in \mathbb{R}^M$.
\item A is a matrix in $\mathbb{R}^{M \times N}$, $M<N$.
\item $A^T$ is the transposed matrix of $A$.
\item For a matrix $A\in \mathbb{R} ^ {M \times N}$, the singular value decomposition (SVD) of $A$ is noted $A=U_A\Sigma(A)V_A^*$.
\item For a matrix $A\in \mathbb{R} ^ {M \times N}$, we denote $\|A\|$ the spectral norm of A defined as
$$
\|A\|= \sigma(A)
$$
where $\sigma(A)$ is the largest singular value of $A$
\item If not stated otherwise, the vector $x\in \mathbb{R}^N$.

\item The indicator function $\iota_{x\in X}$ is defined for $X\subset \mathbb{R}^N$ as 
$$
\iota_{x\in X}(x)=\begin{cases} +\infty \text{ if } x\notin X\\
0 \text{ if } x\in X
\end{cases}
$$

\item The subgradient of the convex function $f$ at point $x$ is the set of vectors $v$ such that 
$$
\forall z\in\,dom(f)\quad f(z)\geq f(x)+v^T(z-x)
$$

\item The normal cone $N_C(x_0)$ of a convex set $C$ in $x_0 \in C$ is defined as
$$N_C(x_0)=\{\eta \in \mathbb{R}^n,<\eta,x-x_0>\leq 0\quad \forall x \in C\}.$$ 

\item  $-{\bf{1}}\leq u \leq {\bf{1}}$ is a component-wise notation, i.e, $\forall \, \, i, \, -1\leq u_i\leq 1$. 
\item $|x|\in \mathbb{R}^N$ is a vector containing the absolute value of each component of the vector $x$.
\end{itemize}

\section{Exact reformulation}
\label{sec:2}

In this section we focus on a reformulation of the $\ell_0$-norm. This reformulation was first introduced in \cite{yuan_sparsity_2016}. The $\ell_0$-norm can be rewritten as a convex minimization problem by introducing an auxiliary variable. 
\begin{lemma}
{\cite[Lemma 1]{yuan_sparsity_2016}}
For any $x\in \mathbb{R}^N$
\begin{equation}
    \|x\|_0 = \min_{-{\bf{1}}\leq u \leq {\bf{1}}} \|u\|_1 \text{ s.t } \|x\|_1= <u,x>
    \label{eq:norm0}
\end{equation}

\end{lemma}
\begin{proof}
We consider first the problem
\begin{equation}
    \min_{-{\bf{1}}\leq u \leq {\bf{1}}} \|u\|_1 \textit{ s.t. } |x_i|=u_ix_i \,\,\,\,  \forall i
\end{equation}
The equality constraint $|x_i|=u_ix_i$ and $-1\leq u_i\leq 1$ yields 
\begin{equation}
   \hat{u}_i\begin{cases}
    =1 \text{ iff } x_i>0\\
    =-1 \text{ iff } x_i<0\\
    \in [-1,1] \text{ iff } x_i=0
    \end{cases}
\end{equation}
As we minimize $\|u\|_1$, if $x_i=0$ then $\hat{u}_i=0$. We have that $\|\hat{u}\|_1=\|x\|_0$. Furthermore, since $u\in [-1,1]$, we have $|x_i|-u_ix_i\geq 0\, \forall i$. So the constraint $|x_i|=x_i u_i \,\forall \, i$ is equivalent to $\sum_i |x_i|=\sum_i x_iu_i$ which is exactly  $\|x\|_1=<x,u>$.
\qed \end{proof}

The introduction of the auxiliary variable increases the dimension of the problem, but the non-convex and non-continuous $\ell_0$-norm can now be written as a {\textit{convex}}  minimization problem. In this paper, we study the $\ell_2-\ell_0$ penalized and constrained problems using the reformulation of the $\ell_0$-norm.  We also add non-negativity constraint to the $x$ variable as it is usually used as a priori in imaging problems. The two problems can be written as a general problem  defined as:

\begin{equation}
    \min_{x,u}  \frac{1}{2}\|Ax-d\|^2+ I(u) +\iota_{\cdot\geq 0}(x) \text{ s.t. } \|x\|_1=<x,u>  \label{eq:Gammel}
\end{equation}
where $I(u)$ is in the case of the constrained problem (\ref{eq:sparseconst}):
\begin{equation}  \label{eq:a}
I(u)=\begin{cases}
0 \text{ if }\|u\|_1 \leq k \text{ and } \forall \, \, i, \, -1\leq u_i\leq 1\\
\infty \text{ otherwise }
\end{cases} 
\end{equation}

and for the penalized problem (\ref{eq:sparspenalty}):
\begin{equation} \label{eq:b}
  I(u)=\begin{cases}
\lambda \|u\|_1 \text{ if } \forall \, \, i, \, -1\leq u_i\leq 1\\
\infty \text{ otherwise }
\end{cases}  
\end{equation}
We note the $\mathcal{S}= \{(x,u); \|x\|_1=<x,u>\}$, and we define the functional $G$ as
\begin{equation}
   G(x,u)=\frac{1}{2}\|Ax-d\|^2+ I(u)+\iota_{\cdot\geq 0}(x)+\iota_{\cdot \in \mathcal{S}}(x,u) \label{eq:G}
\end{equation}

The functional (\ref{eq:G}) is still  non-convex  due to the equality constraint, but it is biconvex: the minimization of (\ref{eq:G}) with respect to $x$ while $u$ is fixed is convex, and conversely.  However, the minimization of such a function is hard because of the equality constraint which is non-convex. We can relax this constraint by introducing a penalty term, $\rho(\|x\|_1-<x,u>)$. This is based on the method of Lagrange Multipliers. Note that it is not necessary to add the absolute value  to this penalty term since $\forall\, \, i\,,|u_i|\leq 1$  and therefore the penalty term is never negative.   

We introduce a Lagrangian cost function  $G_\rho(x,u): \mathbb{R}^{N}\times\mathbb{R}^N \rightarrow \mathbb{R}$  defined as
\begin{equation}
G_\rho (x,u) =  \frac{1}{2}\|Ax-d\|^2+I(u) +\iota_{\cdot\geq 0}(x) +\rho(\|x\|_1-<x,u>) \label{eq:Grho}
\end{equation}
In this paper we are focusing on exact penalty methods, such that a local or global minimizer of (\ref{eq:Grho}) leads to a local or global minimizer of the initial problem (\ref{eq:G}). The following theorem ensures this.  

\begin{theorem}[Constrained form]\label{theo:vikC}
Assume that $\rho>\sigma(A)\|d\|_2$, and $A$ is of full rank. Let $G_\rho$ and $G$ be defined respectively in (\ref{eq:Grho}) and (\ref{eq:G})  with the constrained form  $I(u)$ defined in (\ref{eq:a}). We have:
\begin{enumerate}
    \item If $(x_\rho,u_\rho)$ is a local or global minimizer of $G_\rho$, then $(x_\rho,u_\rho)$ is a local or global minimizer of $G$.
    \item If $(\hat{x},\hat{u})$ is a global minimizer of $G$, then $(\hat{x},\hat{u})$ is a global minimizer of $G_\rho$.
\end{enumerate}
\end{theorem}
Two lemmas are needed in order to proof Theorem \ref{theo:vikC}. The complete proofs of these lemmas require three other lemmas (Lemma \ref{lem:normA}, Lemma \ref{lem:wis}, and Lemma \ref{lem:uconst}) stated in the Appendix. 

\begin{lemma}\label{lem:subPro}
Let $\rho>\sigma(A)\|d\|_2$. Let $(x_{\rho},u_{\rho})$ be a local or global minimizer of $ G_\rho(x,u):=\frac{1}{2}\|Ax-d\|^2+ I(u) +\rho(\|x\|_1-<x,u>)$ with $I(u)$ defined as in (\ref{eq:a}) or (\ref{eq:b}). Let $\omega= \{i\in \{1,\dots,N\}; (u_{\rho})_i=0\}$. Then  $(x_{\rho})_i=0\, \forall i \in \omega$
\end{lemma}
\begin{proof}
Let $J$ denote the set of indices: $J=\{1,\dots, N\} \backslash\omega$.  If $(x_{\rho},u_{\rho})$ is a local or global minimizer of $G_\rho$ then $\forall (x,u)\in \mathcal{N}((x_{\rho},u_{\rho}),\gamma)$, where $\mathcal{N}((x_{\rho},u_{\rho}),\gamma)$ denotes a neighborhood of $(x_{\rho},u_{\rho})$ of size $\gamma$, we have
\begin{align*}
    \frac{1}{2}\|Ax_{\rho}-d\|^2+\iota_{\cdot\geq 0}(x_{\rho})+I(u_{\rho})+&\rho(\|x_{\rho}\|_1-<x_{\rho},u_{\rho}>)\leq\\
    &\frac{1}{2}\|Ax-d\|^2+\iota_{\cdot\geq 0}(x)+I(u)+\rho(\|x\|_1-<x,u>)
\end{align*}

By choosing $u=u_{\rho}$ and $x=\tilde{x}$ with  $\tilde{x}_J=(x_{\rho})_J$ and $\tilde{x}_\omega=x_\omega$, with $(x_\omega,(u_\rho)_\omega) \in \mathcal{N}(((x_{\rho})_\omega,(u_{\rho})_\omega),\gamma)$, we have 
\begin{equation}
 \frac{1}{2}\|Ax_{\rho}-d\|^2+\iota_{\cdot\geq 0}(x_{\rho})+\rho\|(x_{\rho})_\omega\|_1\leq \frac{1}{2}\|A\tilde{x}-d\|^2+\iota_{\cdot\geq 0}(\tilde{x})+\rho\|x_\omega\|_1   \label{eq:x0s}
\end{equation}

We want to show that $(x_\rho)_\omega$ is zero. We have
\begin{align*}
  \|Ax-d\|^2&=\|Ax\|^2+\|d\|^2-2<Ax,d> \\
  &=\sum_i (Ax)_i^2+\|d\|^2-2\sum_i x_i(A^Td)_i\\
  &=\sum_i\left[(\sum_{j\in J}A_{ij}x_j)^2+(\sum_{j\in \omega}A_{ij}x_j)^2\right]+\|d\|^2-\\
  &\quad\quad 2\left[\sum_{i\in J}x_i(A^Td)_i+\sum_{i\in \omega}x_i(A^Td)_i \right]
\end{align*}
Using the above decomposition simplifies (\ref{eq:x0s}), and we have $\forall\,\,\, x_\omega$: 
\begin{align*}
   \frac{1}{2}\sum_i\left(\sum_{j\in \omega} A_{ij}(x_{\rho})_J\right)^2-&\sum_{i\in \omega} (x_{\rho})_i(A^Td)_i + \rho\|(x_{\rho})_\omega\|_1 + \iota_{\cdot\geq 0}(x_{\rho}) \leq\\
   &\frac{1}{2}\sum_i\left(\sum_{j\in \omega} A_{ij}x_j\right)^2-\sum_{i\in \omega} x_i(A^Td)_i + \rho\|x_\omega\|_1+ \iota_{\cdot\geq 0}(x_{\omega})
\end{align*}

Thus $(x_{\rho})_\omega$ is a solution of 
$$
\argmin_{x_\omega}  \frac{1}{2}\sum_i\left(\sum_{j\in \omega} A_{ij}x_j\right)^2-\sum_{i\in \omega} x_i(A^Td)_i + \rho\|x_\omega\|_1 \iota_{\cdot\geq 0}(x_{\omega})
$$
or, equivalently solution of
\begin{equation}
  \argmin_{x_\omega}  \frac{1}{2}\|A_{\omega}x_\omega-d\|^2 + \rho\|x_\omega\|_1+ \iota_{\cdot\geq 0}(x_{\omega})
\end{equation}
where $A_{\omega}$ is the $M\times \#\omega$ submatrix of $A$ composed by the columns indexed by $\omega$ of $A$.   With Lemma \ref{lem:normA} (see Appendix), we have that $\sigma(A)\geq \sigma(A_{\omega})$ and if $\rho>\sigma(A)\|d\|_2$  we can apply Lemma \ref{lem:wis} (see Appendix) with $w$ a vector composed of $\rho$. We conclude that $(x_{\rho})_\omega=0$. 
\qed\end{proof}

\begin{lemma}\label{lem:const}
If $\rho>\sigma(A)\|d\|_2$, let $(x_\rho,u_\rho)$ be a local or global minimizer  of 
$$\argmin_{x,u} \frac{1}{2}\|Ax-d\|^2 +\iota_{\cdot\geq 0}(x)+\rho(\|x\|_1-<x,u>) + I(u)$$
with $I(u)$ defined as in (\ref{eq:a}), that is, the constrained form. Then $\|x_{\rho}\|_1-<x_{\rho},u_{\rho}>=0$.
\end{lemma}
\begin{proof}
From Lemma \ref{lem:uconst} (see Appendix), we have that $(u_{\rho})_i(x_{\rho})_i=|(x_{\rho})_i| \forall\,\, i \in J$, and $(u_{\rho})_i=0 \, \forall i \in \omega$.   It suffices to prove $(x_{\rho})_i=0\, \forall i \in \omega$.  For that we use Lemma \ref{lem:subPro},  and conclude that $(x_{\rho})_\omega=0$.
\qed \end{proof}
With the two above lemmas we can prove Theorem \ref{theo:vikC}

\begin{proof}
We start by proving the first part of the theorem.
 Let $(x_\rho,u_\rho)$ be a local minimizer of $G_\rho$, with $I(u)$ on the constrained form, that is, defined as in (\ref{eq:a}). Let $\mathcal{S}= \{(x,u); \|x\|_1=<x,u>\}$.
 If $\rho>\sigma(A)\|d\|_2$ then, from Lemma \ref{lem:const}, 
 $$(x_\rho,u_\rho) \text{ verifies } \|x_\rho\|_1=<x_\rho,u_\rho>.$$
 Furthermore, from the definition of a minimizer, we have
 $$
  G_\rho (x_\rho,u_\rho) \leq G_\rho(x,u)\, \, \forall (x,u)\in \mathcal{N}((x_{\rho},u_{\rho}),\gamma)
 $$
and so we have
$$
 G_\rho (x_\rho,u_\rho) \leq G_\rho(x,u) \, \, \forall (x,u)\in \mathcal{N}((x_{\rho},u_{\rho}),\gamma)\cap \mathcal{S}
$$
Since  $\forall (x,u) \in \mathcal{S}, G_\rho (x,u)=G(x_\rho,u_\rho)$, we have
\begin{equation}
     G (x_\rho,u_\rho)\leq G(x,u)  \, \, \forall (x,u)\in \mathcal{N}((x_{\rho},u_{\rho}),\gamma)\cap \mathcal{S}
\end{equation}
By the definition, $(x_\rho,u_\rho)$ is also a local minimizer of $G$.

Now we  prove part 2 of Theorem \ref{theo:vikC}. 

Let $(\hat{x},\hat{u})$ be a global minimizer of $G$. We necessarily have  $\|\hat{x}\|_1=<\hat{x},\hat{u}>$. First, we show that
$$G_\rho(\hat{x},\hat{u})\leq \min G_\rho(x,u).$$ 
This can be shown by contradiction. Assume the opposite, and denote $(x_\rho,u_\rho)$ a global minimizer of $G_\rho$. We then have 
\begin{equation}\label{eq:thecontra}
   G_\rho(\hat{x},\hat{u})> \min G_\rho(x,u)=G_\rho(x_\rho,u_\rho) 
\end{equation}

Lemma \ref{lem:const} shows that $\|x_\rho\|_1=<x_\rho,u_\rho>$, so $G_\rho(x_\rho,u_\rho)=G(x_\rho,u_\rho)$ and we have 
$$
   G(\hat{x},\hat{u})=G_\rho(\hat{x},\hat{u})> \min G_\rho(x,u)=G_\rho(x_\rho,u_\rho) =G(x_\rho,u_\rho)
$$
and more precisely, $G(\hat{x},\hat{u})>G(x_\rho,u_\rho)$ which is not possible, since $(\hat{x},\hat{u})$ is a global minimizer of $G$. 

We therefore have  shown that $G_\rho(\hat{x},\hat{u})\leq \min G_\rho(x,u)$, and we have
$$
 G_\rho(\hat{x},\hat{u})\leq \min G_\rho (x,u)\leq  G_\rho(x,u)\quad \forall (x,u)
$$
$(\hat{x},\hat{u})$ is thus a global minimizer of $G_\rho$.
\qed
\end{proof}

\begin{theorem}[Penalized form]\label{theo:vikP}
Assume that $\rho>\sigma(A)\|d\|_2$, and $A$ is of full rank. Let $G_\rho$ and $G$ be defined respectively in (\ref{eq:Grho}) and (\ref{eq:G}) with on the penalized form with $I(u)$ defined in (\ref{eq:b}). We have:
\begin{enumerate}
    \item If $(x_\rho,u_\rho)$ is a local or global minimizer of $G_\rho$, then we can construct  $(x_\rho,\tilde{u}_\rho)$ which is a local or global minimizer of $G$.
    \item If $(\hat{x},\hat{u})$ is a global minimizer of $G$, then $(\hat{x},\hat{u})$ is a global minimizer of $G_\rho$.
\end{enumerate}
\end{theorem}
For the proof, we need two lemmas, Lemma \ref{lem:subPro} which is already presented and the following lemma.

\begin{lemma}\label{lem:Penal}
Let $(x_\rho,u_\rho)$ be a local or a global minimizer of $G_\rho$ for the penalized form ($I(u)$ defined as in (\ref{eq:b})). If $\rho>\sigma(A)\|d\|_2$ then $\forall \, i$ such that $ (u_{\rho})_i=0$ we have $(x_{\rho})_i =0$ 
\end{lemma}
\begin{proof}
From Lemma \ref{lem:upenal} (see Appendix), we have that $(u_{\rho})_i=0$ iff $(x_{\rho})_i\in ]-\frac{\lambda}{\rho},\frac{\lambda}{\rho}[$.  We denote $\omega$ the set of indices where $u_{\rho}=0$, and we can apply  Lemma \ref{lem:subPro},  and conclude that $(x_{\rho})_{\omega}=0$.

\qed \end{proof}
\begin{remark}\label{rem:const}
If $\rho>\sigma(A)\|d\|_2$, note that the cost function $G_\rho$ with minimizers $(x_{\rho},u_{\rho})$  is constant on $|(x_{\rho})_i|=\frac{\lambda}{\rho}$ and $|(u_{\rho})_i|\in [0, 1]$.  
\end{remark}
\begin{remark} \label{rem:x}
In the case of the penalized form, the minimizers ($x_{\rho},u_{\rho}$) of $G_\rho$ with $\rho>\sigma(A)\|d\|_2$ may be such that $<x_{\rho},u_{\rho}> \neq \|x_{\rho}\|_1$. This may only happen if $|(x_{\rho})_i|=\frac{\lambda}{\rho}$. 
\end{remark}
\begin{remark}\label{rem:penalRem}
If $\rho>\sigma(A)\|d\|_2$. From Remark \ref{rem:const}, from a minimizer $(x_{\rho},u_{\rho})$ of $G_\rho$, we can construct a minimiser $(x_{\rho},\tilde{u}_\rho)$ of $G_\rho$ such that $\|x_{\rho}\|_1=<x_{\rho},\tilde{u}_\rho>$. This can be done by denoting $Z$, the set of indices such that $0<|(u_{\rho})_i|<1$. If $Z$ is non-empty, we have $<x_\rho,u_\rho>\neq \|x_\rho\|_1$. From Remark \ref{rem:x}, $|(x_{\rho})_i|=\frac{\lambda}{\rho} \forall i \in Z$. Take $\tilde{u}_{\rho i} = sign(x_i) \,\,\, \forall i \in Z$ and $\tilde{u}_{\rho i} = (u_{\rho})_i \, \forall i \notin Z$, then $<x_\rho,\tilde{u}_\rho>=\|x_\rho\|_1$. Furthermore, $(x_\rho,\tilde{u}_\rho)$ is a minimizer of $G_\rho$ after Remark \ref{rem:const} and the fact that $G_\rho(x_\rho,u)$ is convex with respect to $u$. 
\end{remark}

With Lemma \ref{lem:Penal} and the above remarks, we can prove Theorem \ref{theo:vikP}. 
\begin{proof}
We start by proving the first part of the theorem.
 Given $(x_\rho,u_\rho)$ a local or global minimizer of $G_\rho$, with $I(u)$ on the penalized form, that is, defined as in (\ref{eq:a}). Let $\mathcal{S}$ denote the space where $\|x\|_1=<x,u>$.
 If $\rho>\sigma(A)\|d\|_2$ then, from remark \ref{rem:penalRem}, we can construct $(x_\rho,\tilde{u}_\rho)$ such that 
 $$(x_\rho,\tilde{u}_\rho) \text{ verifies } \|x_\rho\|_1=<x_\rho,\tilde{u}_\rho>.$$
 Furthermore, from the definition of a minimizer, we have
 $$
 G_\rho (x_\rho,\tilde{u}_\rho)\leq  G_\rho(x,u) \, \, \forall (x,u)\in \mathcal{N}((x_{\rho},\tilde{u}_\rho),\gamma)
 $$
and so we get
$$
  G_\rho (x_\rho,\tilde{u}_\rho)\leq  G_\rho(x,u) \, \, \forall (x,u)\in \mathcal{N}((x_{\rho},\tilde{u}_\rho),\gamma)\cap \mathcal{S}
$$
Since  $\forall (x,u) \in \mathcal{S}, G_\rho (x,u)=G(x_\rho,u_\rho)$, we obtain
\begin{equation}
    G (x_\rho,\tilde{u}_\rho) \leq  G(x,u)\, \, \forall (x,u)\in \mathcal{N}((x_{\rho},\tilde{u}_\rho),\gamma)\cap \mathcal{S}
\end{equation}
Then, $(x_\rho,\tilde{u}_\rho)$ is also a local minimizer of $G$.

The second part of Theorem \ref{theo:vikP} can be proved as in the proof of Theorem \ref{theo:vikC}.
\qed
\end{proof}

Theorem \ref{theo:vikC} and \ref{theo:vikP} show that, for a given $\rho$,  minimizing (\ref{eq:Grho}) is equivalent in terms of minimizers as minimizing (\ref{eq:G}).
The results in this section are similar to \cite{yuan_sparsity_2016}. In their paper, they, instead of working with the square norm, work with a Lipschitz continuous function $f$. We were inspired by their work to extend it to the square norm.  They have a theorem equivalent to Theorem \ref{theo:vikC}, but the lower bound for $\rho$ is less sharp. Furthermore, the paper \cite{yuan_sparsity_2016} does not tackle the penalized sparsity problem and thus has not a theorem equivalent to Theorem \ref{theo:vikP}.   

Although  $G_{\rho}(x,u)$ in (\ref{eq:Grho}) is non-convex,  the formulation is biconvex, i.e, the functional is convex with respect to $x$ when $u$ is constant and conversely.  With that in mind, we propose in the next section an algorithm to minimize (\ref{eq:Grho}).

\section{A minimization algorithm}

The functional $G_\rho$ has two interesting particularities. The first is that the non-convexity of $G_\rho$  is due to the coupling term $<x,u>$. $G_\rho$ is therefore convex when the penalty parameter $\rho$ equals to zero. This inspires the idea of an algorithm to minimize $G_\rho(x,u)$. The minimization starts with a $\rho^0$ small and minimizes $G_{\rho^0}(x^0,u^0)$. For each iteration, the penalty parameter, $\rho$, increases and the solution of the previous iteration are used as initialization for the next minimization. This method will hopefully give a good initialization for the final minimization, that is when $\rho$ is according to Theorem \ref{theo:vikC} and Theorem \ref{theo:vikP}. The second interesting property of functional $G_\rho$ is the biconvexity. Minimization by blocks is therefore interesting since with respect to each block (that is either $x$ or $u$), the problem is convex.  With this in mind, and following \cite{yuan_sparsity_2016}, we propose the following algorithm. 

\begin{algorithm}[H]
  \Entree{\par\leftskip=1.5em
    $\rho^0$ small ;
    }\BlankLine
    
    \Init{\par\leftskip=1.5em
    $x^0=\mathbf{0}\in \mathbb{R}^N$;  $u^0=\mathbf{0}\in \mathbb{R}^N$; $p=0$;}\BlankLine
    
  \rep{\par\leftskip=1.5em
  Solve problem $G_{\rho^p}$
    \begin{equation}
        \{x^{p+1},u^{p+1}\}\in \argmin G_{\rho^p}(x^p,u^p)
        \label{eq:usePalm}
    \end{equation}
 Update the penalty term
 \begin{equation}
     \rho^{p+1}=\min(\sigma(A)\|d\|_2,2\rho^{p})
 \end{equation}
  }\BlankLine
  
  \until{
  $\rho^{p+1}=\sigma(A)\|d\|_2$   
 }\BlankLine
  \Sortie{
  	$x^{p+1}$}
    
  \caption{Biconvex minimization}
   \label{alg:MPECEPM}
  \end{algorithm}

The minimization of (\ref{eq:usePalm}) is done by using the Proximal Alternating Minimization algorithm (PAM) \cite{attouch_proximal_2008} which ensures convergence to a critical point. The  PAM minimizes functions on the form 
\begin{equation}
    L(x,u)=f(x)+g(u)+Q(x,u)
\end{equation}
In our case, we have, $f(x)=\frac{1}{2}\|Ax-d\|^2+\rho\|x\|_1+ \iota_{\cdot\geq 0}(x)$, $g(u)=I(u)$ and $Q(x,u)=-\rho<x,u>$.
PAM has the following outline 
$$
\begin{cases} \text{Repeat}\\
x^{s+1}\in \argmin_x \left \{G_\rho(x,u^s) +\frac{1}{2c^s}\|x-x^s\|_2^2 \right\}\\
u^{s+1}\in \argmin_u \left \{G_\rho(x^{s+1},u) +\frac{1}{2b^s}\|u-u^s\|_2^2 \right\}\\
\text{Until convergence}
\end{cases}
$$
$c^s$ and $b^s$ add strict convexity to each block, and $c^s,b^s$ are bounded from below and above.  

In the following section we develop the minimization schemes for (\ref{eq:usePalm}) in the case of the constrained problem ($I(u)$ defined as in (\ref{eq:a})) respectively the penalized problem ($I(u)$ defined as in (\ref{eq:b})). We recall the minimization of $G_\rho$  is
\begin{equation}
    \argmin_{x,u} \frac{1}{2}\|Ax-d\|^2+I(u)+\rho(\|x\|_1-<x,u>)+\iota_{\cdot\geq 0}(x) \label{eq:Constrho}
\end{equation}
where $I(u)$ is defined in (\ref{eq:a}) or in (\ref{eq:b}).

\subsection{The minimization with respect to $x$.}\label{sec:mimix}
The minimization with respect to $x$ using  PAM  is 
$$
x^{s+1}\in\argmin_x \frac{1}{2}\|Ax-d\|^2 + \rho(\|x\|_1 - <x,u^s>) +\frac{1}{2c^s}\|x-x^s\|_2^2+\iota_{\cdot\geq 0}(x)
$$
which can be rewritten as 
$$
x^{s+1}\in\argmin_x \frac{1}{2}\|Ax-d\|^2 + \frac{1}{2c^s}\|x-(x^s+\rho c^su^s)\|^2 + \rho\|x\|_1+\iota_{\cdot\geq 0}(x)
$$
This problem can be solved using the FISTA algorithm \cite{beck_fast_2009}. The algorithm works with a functional $F(x)=f(x)+g(x)$ where $f$ is a smooth convex function with a Lipschitz continuous gradient $L(f)$. $g$ is a continuous  convex function and possibly non-smooth.  In our case we have
\begin{align}
    f(x)&=\frac{1}{2}\|Ax-d\|^2 + \frac{1}{2c^s}\|x-(x^s+\rho c^su^s)\|^2\\
    g(x)&=\rho\|x\|_1 +\iota_{\cdot\geq 0}(x)
\end{align}

The proximal operator of $g(x)$ is the soft thresholding with positivity constraint
$$
\text{prox}_{\frac{g}{L(f)}}(x)=\begin{cases}
x_i-\frac{\rho}{L(f)} \text{ if } x_i>\frac{\rho}{L(f)}\\
0 \text{ if } x_i \leq \frac{\rho}{L(f)}]
\end{cases} 
$$

\subsection{The minimization with respect to $u$}
In this section we study how to find a solution to the following convex minimization problem 
$$
u^{s+1}=\argmin_{u} \frac{1}{2b^s}\|u-u^s\|_2^2-\rho<x^{s+1},u> +I(u)
$$
The above problem can be rewritten as 
\begin{equation}\label{eq:minU}
    u^{s+1}=\argmin_u \frac{1}{2b^s}\|u-(u^s+\rho b^sx^{s+1})\|^2 +I(u) 
\end{equation}

and for simplicity we denote $z=u^s+\rho b^sx^{s+1}$. 
\subsubsection{The constrained minimization of $u$}
In this section we work with the constrained formulation of $G_\rho$. Then the minimization problem (\ref{eq:minU}) can be simplified and written as 
$$
u^{s+1}=\argmin_u \frac{1}{2}\|u-z\|^2 \text{ s.t. } \|u\|_1 \leq k  \text{ and } \forall \, \, i, \, -1\leq u_i\leq 1
$$

Since the minimizer of $\argmin_u \frac{1}{2}\|u-z\|^2$ is reached for $u=z$, we can write $u^{s+1}=sign(z)\argmin_u \frac{1}{2}\|u-|z|\|^2$. Furthermore, since the $\|\cdot\|_1$ is invariant with respect to the sign, we can rewrite the minimization problem as 
$$
|u^{s+1}|=\argmin_{u} \frac{1}{2}\|u-|z|\|^2 \text{ s.t. } \|u\|_1 \leq  k \text{ and } \forall \, \, i, \, 0\leq u_i\leq 1
$$
and then $u^{s+1}=sign(z)|u^{s+1}|$.
This minimization problem is a variant of the  knapsack problem  which can be solved using classical minimization schemes such as \cite{doi:10.1155/S168712000402009X} : 
\begin{align*}
|u^{s+1}|= \argmin_u &\frac{1}{2}<u,u>-<u,|z|> \\
&\text{ s.t. } \left(\sum_i u_i \right)\leq k\\
& \text{ and } \forall \, \, i, \, 0\leq u_i\leq 1
\end{align*}

\subsubsection{The penalized minimization of  $u$}
The minimization of (\ref{eq:minU}) with respect to $u$, with $I(u)$ on the penalized form (\ref{eq:b}), can be written as
$$u^{s+1}=\argmin_{u} \lambda\|u\|_1+\frac{1}{2b^s}\|u-z\|^2 +\iota_{-1\leq \cdot \leq 1}(u) $$

\begin{proposition}
The solution $u^{s+1}$ of 
\begin{equation}\label{eq:minUP}
    \argmin_{u} \lambda\|u\|_1+\frac{1}{2}\|u-z\|^2 +\iota_{-1\leq \cdot \leq 1}(u)
\end{equation}
is reached for 
$$
(u^{s+1}_{\rho})_i=
\begin{cases}
1 \text{ if } z_i\in [1+\lambda b^s, \infty[\\
z_i-\lambda b^s \text{ if } z_i\in ]\lambda b^s,1+\lambda b^s[\\
0 \text{ if } z_i\in\lambda b^s[-1,1]\\
z_i+\lambda b^s \text{ if } z_i\in ]-1-\lambda b^s,-\lambda b^s[  \\
-1 \text{ if } z_i\in ]-\infty,-1-\lambda b^s ]
\end{cases}$$
\end{proposition}
\begin{proof}

Problem (\ref{eq:minUP}) has a closed form expression which can be found by calculating  the subgradient for the problem (\ref{eq:minUP}) with respect to $u$. Note that the subgradient of the box constraint $\iota_{-1\leq \cdot \leq 1}$ is 0 if $|u_i|< 1$, $[0,\infty[$ if $u_i=1$ and $]-\infty,0]$ if $u_i=-1$. We obtain the following optimal conditions: 
$$
0 \in 
\begin{cases}
\lambda + [0,\infty[+ \frac{1}{b^s}(u^{s+1}_i-z_i) \text{ if } u^{s+1}_i=1\\
\lambda +\frac{1}{b^s}(u^{s+1}_i-z_i) \text{ if } 1>u^{s+1}_i>0\\
\lambda[-1,1] -\frac{1}{b^s}(z_i) \text{ if } u^{s+1}_i=0\\
-\lambda +\frac{1}{b^s}(u^{s+1}_i-z_i) \text{ if } -1<u^{s+1}_i<0\\
-\lambda + ]-\infty,0]+ \frac{1}{b^s}(u^{s+1}_i-z_i) \text{ if } u^{s+1}_i=-1\\
\end{cases}
$$
and the optimal solution $u_{\rho}$ is 
$$
(u^{s+1}_{\rho})_i=
\begin{cases}
1 \text{ if } z_i\in [1+\lambda b^s, \infty[\\
z_i-\lambda b^s \text{ if } z_i\in ]\lambda b^s,1+\lambda b^s[\\
0 \text{ if } z_i\in\lambda b^s[-1,1]\\
z_i+\lambda b^s \text{ if } z_i\in ]-1-\lambda b^s,-\lambda b^s[  \\
-1 \text{ if } z_i\in ]-\infty,-1-\lambda b^s ]
\end{cases}$$
\qed
\end{proof}

\section{Application to 2D single-molecule localization microscopy}
In this section, we compare the minimization of the biconvex reformulations to the algorithm Iterative Hard Thresholding \cite{combettes2005signal} where we add the non-negativity constraint to $x$. This  algorithm  performs as well to both formulations (\ref{eq:sparseconst}) and (\ref{eq:sparspenalty}). They are applied to the problem of 2D Single-Molecule Localization Microscopy (SMLM). 

SMLM is a microscopy method which is used to obtain images with a higher resolution than what is possible with normal optical microscopes. This was first introduced in \cite{hess2006ultra, betzig2006imaging,rust2006sub}. Fluorescent microscopy uses molecules that can emit light when they are excited with a laser. The molecules are observed with an optical microscope, and, since the molecules are smaller than the diffraction limit, what is observed is not each molecule, but rather a diffraction disk larger than the molecule. This limits the resolution of the image. SMLM exploits photoactivatable fluorescent molecules, and, instead of activating all the molecules at once as done by other fluorescent microscopy methods, activates a sparse set of fluorescent molecules.   The localization of each molecule with a high precision is possible since the probability of two or more molecules to be in the same diffraction disk is small. The localization becomes harder if the density of emitting molecules is higher. Once each molecule has been precisely localized, they are switched off and the process is repeated until all the molecules have been activated. The total acquisition time may be long when activating few molecules at a time, which is unfortunate as SMLM may be used on living samples which can move during this time. This will lead to a faulty reconstruction. We are, in this paper, interested in high-density acquisitions. 

The localization problem of SMLM can be described as a $\ell_2-\ell_0$ minimization problem such as (\ref{eq:sparspenalty}) and (\ref{eq:sparseconst}) with an added positivity constraint since we reconstruct the intensity of the molecules. The two biconvex formulations can be applied to the SMLM problem. $A $ is the matrix operator that performs a convolution with the Point Spread function and a reduction of dimensions. The molecules are reconstructed on a grid $\in \mathbb{R}^{ML\times ML}$ which is finer than the observed image $\in \mathbb{R}^{M \times M}$, with $L>1$. For a complete lecture on the mathematical model, see \cite{gazagnes2017high}.

We test the algorithms on two datasets, both accessible from the ISBI 2013 challenge \cite{sage2015quantitative}. Both datasets are of high-density acquisitions. The first dataset contains simulated acquisitions, which makes it possible to do a numerical evaluation of the reconstruction. The second dataset contains real acquisitions.  For a complete lecture on the SMLM and the different localization algorithms, see the ISBI-SMLM challenge \cite{sage2015quantitative}.  In Figure \ref{fig:donneesISBI} are three of the 361 acquisitions of the simulated dataset. We apply the localization algorithms to each acquisition, and the results of the localization of the 361 acquisitions yields one super-resolution image. 
\begin{figure}[]
		\centering
		\includegraphics[width=3.2cm]{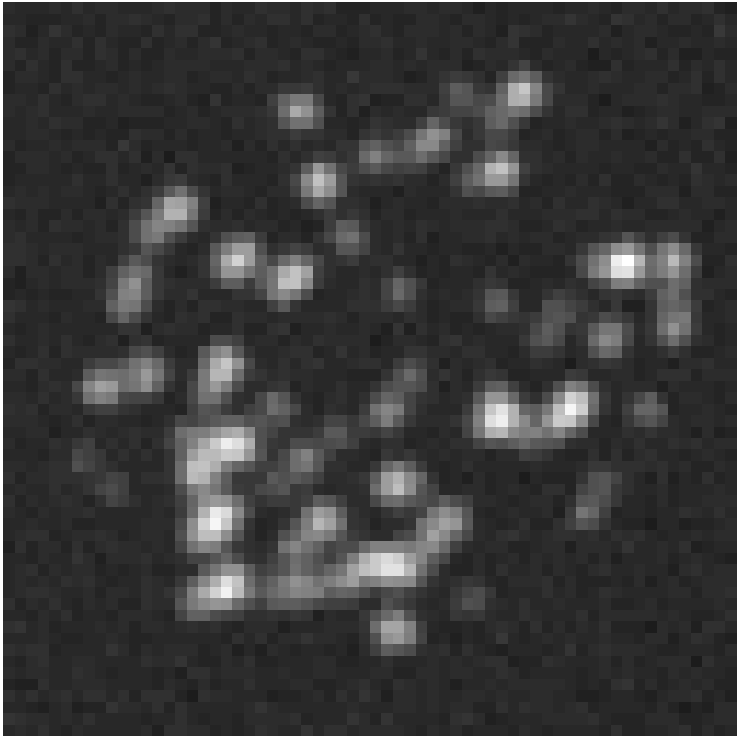}\hspace{0.2cm}
		\includegraphics[width=3.2cm]{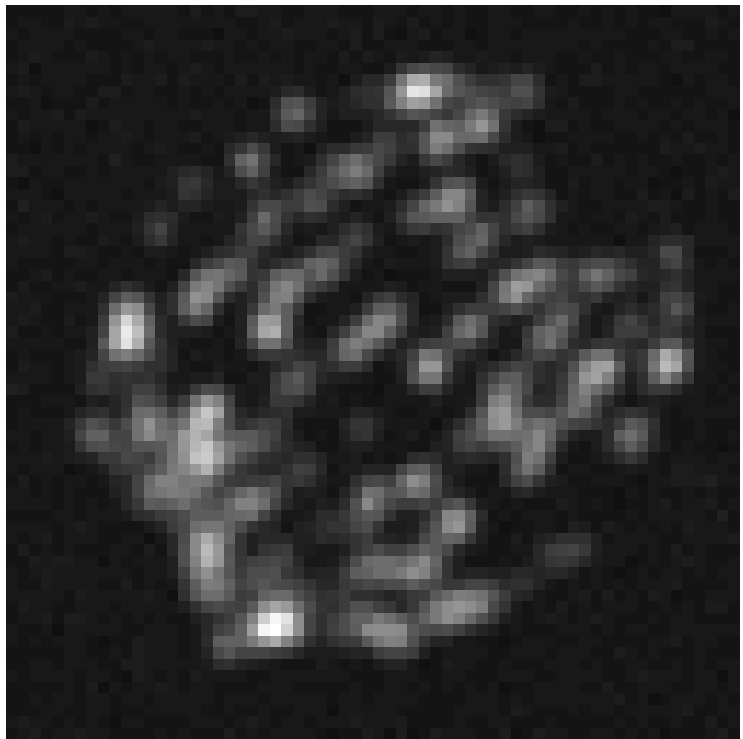}\hspace{0.2cm}
		\includegraphics[width=3.2cm]{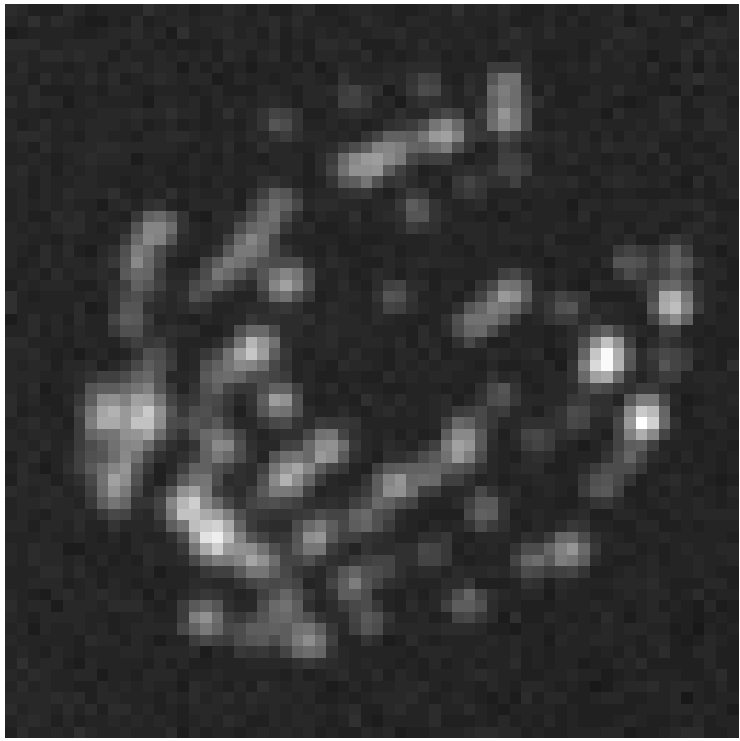}
		\caption{\label{fig:donneesISBI} Simulated images  (among the 361 simulated high density acquisitions). }
	\end{figure}
We use the Jaccard index in order to perform a numerical evaluation of the reconstructions. The Jaccard index evaluates only the localization of the reconstructed molecules (see \cite{sage2015quantitative}). The Jaccard index is the ratio between the correctly reconstructed (CR) molecules and the sum of CR-, false negatives (FN)- and false positives (FP) molecules. The index is 1 for a perfect reconstruction, and the lower the index, the poorer the reconstruction. The Jaccard index includes a tolerance of error in its calculations of correctly reconstructed molecules. 
\begin{equation*}
    Jac=\frac{CR}{CR+FP+FN}.
    \label{eq:Jaccard}
\end{equation*}

\subsection{Results of the simulated dataset}
The ISBI simulated dataset represents 8 tubes of 30 nm diameter. The acquisition is of the size of $64\times 64$ pixels where each pixel is of size $100 \times 100$ nm$^2$. The Point Spread Function (PSF) is modeled by a Gaussian function where the Full Width at Half Maximum (FWHM) is 258.21 nm. In total there are 81 049 molecules on a total of 361 images. 

We localize the molecules with a higher precision on a $256\times 256$ pixel image, where the size of each pixel is $25\times 25$ nm$^2$. As an optimization problem, this is equivalent to reconstruct $x\in \mathbb{R}^{ML\times ML}$ for an acquisition $d\in \mathbb{R}^{M\times M}$, where $M=64$ and $L=4$.  The center of the pixel is used to estimate the position of the molecule.

We set $k$, the maximum number of molecules the algorithm reconstructs, equal to 220 for the  constrained problem. This number is the average number of molecules for each acquisition, which we know from the ground truth. Note that in order to observe the reconstruction, we normalize the image, that is, we let the smallest value in the image to be 0, and the largest to be 1. Each pixel has an intensity between 0 and 1, and the brighter the pixel the stronger the intensity.

 We set $\rho=0.1$ for the biconvex algorithms. Note that a smaller $\rho$ could be chosen, but this implies longer computational time.  Both constrained algorithms reconstruct 220 molecules for each acquisition. We choose a $\lambda$ for the two penalized algorithms such that they reconstruct around 220 molecules on average. For the IHT, $\lambda=0.13$ and for the biconvex penalized $\lambda=0.019$. We initialize the IHT with applying the conjugate of the operator $A$ on the acquisition.  The results of the reconstructions are shown in Figure \ref{fig:xsim3}. Both biconvex reformulations reconstruct the tubes thicker than the ground truth.  The two IHT algorithms do not manage to distinguish between two tubes when they are close (see the red case in Figure \ref{fig:xsim3}) compared to the biconvex reformulations. The Jaccard index is shown in Table \ref{tab:sim3}. We observe the low Jaccard index of the IHT constrained algorithm compared to the biconvex  constrained algorithm. This might be surprising since the IHT seems to reconstruct the tubelins with a correct thickness. However, this indicates that IHT reconstruct many molecules of low intensity which are not situated on the tubelins.

\begin{figure}[]
\centering

\includegraphics[width=.95\textwidth]{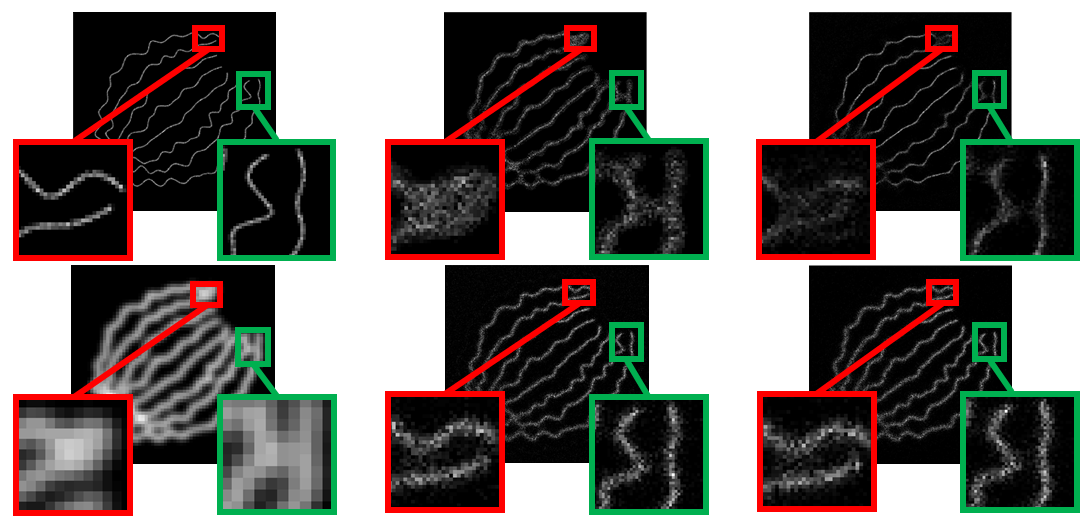}

\caption{Reconstructed of the simulated ISBI dataset, 220 molecules on average. Top: From left to right: Ground Truth, IHT Penalized and IHT Constrained.  Bottom: From left to right: Observed image, Biconvex Penalized and Biconvex Constrained}
\label{fig:xsim3}
\end{figure}

\begin{table}[]
\centering
\begin{tabular}{|l|l|l|l|l|}
\hline & \multicolumn{4}{l|}{Jaccard index (\%)}   \\ \hline
Method - Tolerance (nm) &  50 &  100 &  150 &  200 \\ \hline
IHT - Constrained                                           & \textbf{20.0}                     & 39.4                      & 48.9                      & 54.3                       \\ \hline
Biconvex - Constrained                         & \textbf{20.0}                 & \textbf{48.3 }    & \textbf{61.4}                       & \textbf{67.7}   \\ \hline

IHT - Penalized             &  13.1            &   31.2             &  35.7             & 38.0       \\ \hline
Biconvex - Penalized                           & 18.1                      & 40.0                       & 50.0                      & 54.7             \\ \hline
\end{tabular}
\caption{The Jaccard index with respect to the tolerance and the different algorithms. In bold is the best reconstruction. }\label{tab:sim3}
\end{table}

\subsection{Results of the real dataset}
We compare the algorithms on a high-density dataset of tubulins which are provided from the 2013 ISBI SMLM challenge, where there are 500 acquisitions. Each acquisition is of size $128\times 128$ pixels and each pixel is of size $100 \times 100$ nm$^2$. The FWHM has been previously estimated to be 351.8 nm \cite{chahid2014echantillonnage}. We localize the molecules on a $512\times512$ pixel image, where each pixel is of size $25 \times 25$ nm$^2$. 

In this section, we do not have any beforehand knowledge of the solution, and we set $k=140$ for the biconvex constrained algorithm.  For the biconvex penalized algorithm we set $\lambda=1200$. We choose $\rho= 1$ because of computational time. For the constrained IHT algorithm, we set the constraint $k=100$ and for the penalized we set $\lambda=0.25$. Figure \ref{fig:xsim2} presents the reconstruction. The results are coherent with the results from the simulated dataset. The IHT algorithms reconstruct not as well as the biconvex algorithms, with the penalized version much worse than the constrained version.  

\begin{figure}[]
\centering

\includegraphics[width=.9\textwidth]{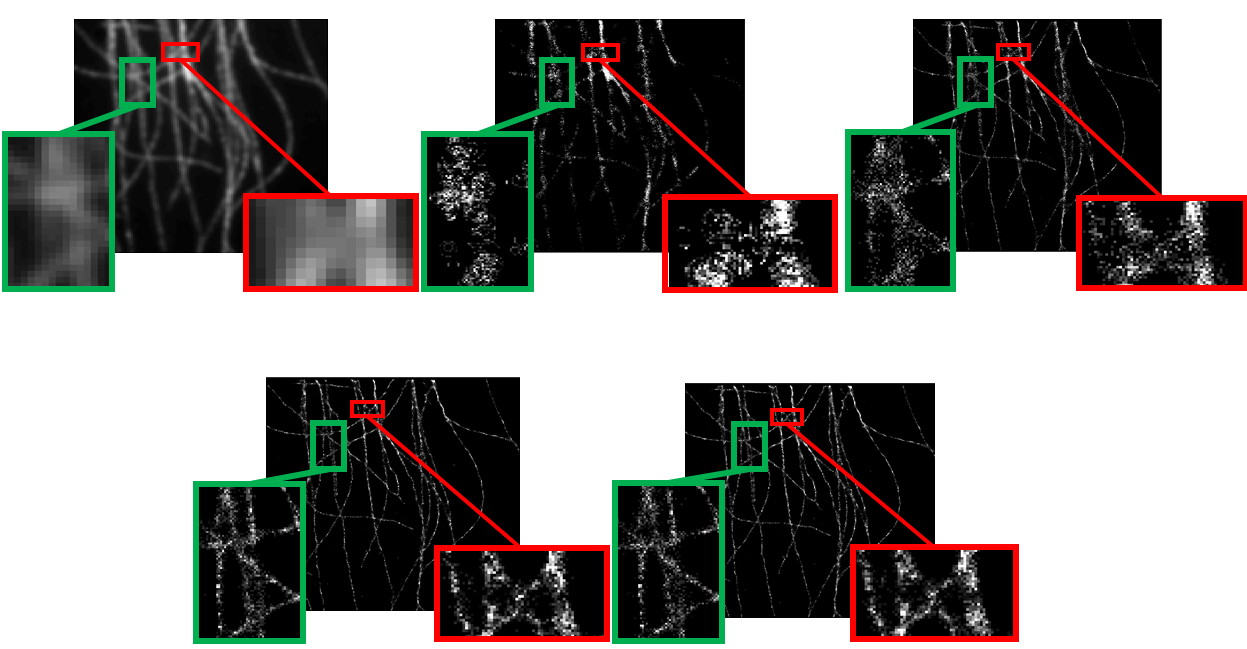}

\caption{ Real ISBI dataset. Top: From left to right: Observed Image, IHT Penalized and IHT Constrained.  Bottom: From left to right:  Biconvex Penalized and Biconvex Constrained}
\label{fig:xsim2}
\end{figure}

\section{Conclusion}
In this paper, we have presented a reformulation of the $\ell_2-\ell_0$ constrained and penalized problems. We have proved in Theorem \ref{theo:vikC} and Theorem \ref{theo:vikP} the exactness of the reformulations, that is, we can from a minimizer of the reformulation obtain a minimizer of the initial problem. Furthermore, both reformulations are biconvex. Using two central properties of the reformulation, we derive a general algorithm in order to minimize the constrained or the penalized reformulation. This algorithm is easy to implement as each step can be decomposed to well-studied problems.  The algorithms are compared to the well-known IHT algorithm on constrained and penalized form. We apply the algorithms to single-molecule localization microscopy and the two biconvex algorithms outperform the IHT algorithms visually and numerically. 

As perspectives, it seems interesting to further investigate the reformulation of the $\ell_0$-norm, and to introduce it with other data-fitting terms.

\section*{Appendix}
\addcontentsline{toc}{section}{Appendix}
\label{app:proof1}
In this Appendix we recall  and  prove some properties that are useful for the proof of Theorem \ref{theo:vikC} and Theorem \ref{theo:vikP}.

\begin{lemma}\label{def:semi}
Let $P\in \mathbb{R}^{N\times l}$ be a semi-orthogonal matrix, that is, a non-square matrix composed of orthonormal columns. Then, $P^TP$ is the identity matrix in $\mathbb{R}^{l\times l}$. 
\end{lemma}

\begin{lemma}\label{lem:normA}
Let $A\in \mathbb{R}^{M\times N}$, let $a_i$ denote the $i$th column of $A$.  Defining $\omega$ to be a set of indices, $\omega\subseteq \{1,\dots ,N\}$. Let the restriction of $A$ to the columns indexed by the elements of $\omega$ be denoted as $A_\omega=(a_{\omega[1]},\dots,a_{\omega[\#\omega]})\in \mathbb{R}^{M\times \#\omega}$.  Then $\|A_\omega\|\leq \|A\|$.
\end{lemma}
\begin{proof}
Note that we can write $A_\omega$ as the product of matrix $A$ and a matrix $P$. We define the vector $e_i\in \mathbb{R}^{M}$, the unitary vector which has zeros everywhere except for the $i^th$ place. The matrix $P \in \mathbb{R}^{N \times \# \omega}$ can be constructed with $e_i\, \forall i \in \omega$. The matrix $P$ is therefore a semi-orthonormal matrix. The spectral norm of the matrix $P$ is 1, as $P^TP$ is the identity matrix (from Lemma \ref{def:semi}). The norm $A_\omega$ can  be written as
\begin{equation}
    \|A_\omega\|=\|AP\|\leq \|A\|\|P\|=\|A\|
\end{equation}
\qed\end{proof}

\begin{lemma}{ [Pshenichnyi-Rockafellar lemma]\cite[Theorem 2.9.1]{zalinescu2002convex} }
Assume $g$ is a  proper lower semi-continuous convex function. Let $C$ be a convex set, such that $int (C) \cap dom (g)\neq \emptyset$. Then
$$
\hat{x}=\argmin_{x\in C} g(x)\Leftrightarrow\textbf{0}\in \partial g(\hat{x})+N_C(\hat{x})
$$
where $N_C$ is the normal cone of the convex set $C$. 
\end{lemma}

\begin{lemma}\label{lem:bound2}
Given the problem 
\begin{equation} \label{eq:bound2}
    \argmin_{x} \frac{1}{2}\|Ax-d\|^2 +<w,|x|>
\end{equation}
where $A$ $\mathbb{R}^{M \times N}$ is a full rank matrix and  $w$ a non-negative vector. $|x|$ is a vector which contains the absolute value of each component of $x$. Let $\hat{x}$  be a solution of  problem (\ref{eq:bound2}). 
Then $\|A\hat{x}-d\|_2$ is bounded independently of $w$ and 
\begin{equation}\label{eq:Resbound2}
    \|A\hat{x}-d\|\leq \|d\|
\end{equation}
\end{lemma}
\begin{proof}
Let $\hat{x}$ be the solution of $\argmin_{x} \frac{1}{2}\|Ax-d\|^2 +<w,|x|>$, then we have $\forall x\in \mathbb{R}^N$
\begin{equation}
    \frac{1}{2}\|A\hat{x}-d\|^2+<w,|\hat{x}|>\leq\frac{1}{2}\|Ax-d\|^2+<w,|x|>.
\end{equation}
In particular, by choosing $x=0$ we have:  

\begin{align}
    \frac{1}{2}\|A\hat{x}-d\|^2+<w,|\hat{x}|>\leq      \frac{1}{2}\|d\|^2.
\end{align}

The term $<w,|\hat{x}|>$ is always non-negative as $w$ is a non-negative vector, therefore we have

$$
\frac{1}{2}\|A\hat{x}-d\|^2\leq \frac{1}{2}\|d\|^2
$$
and so
$$
\|A\hat{x}-d\|\leq \|d\|.
$$

\qed \end{proof}

\begin{lemma}\label{lem:wis}
Let $f(x)=\frac{1}{2}\|Ax-d\|_2^2+<w,|x|>+\iota_{\cdot\geq 0}(x)$, $A$ be a full rank matrix and $w$ is a non-negative vector. We have the following result: If $w_i> \sigma(A)\|d\|_2$ then the optimal solution of the following optimization problem:
\begin{equation}
 \hat{x}=\argmin_{x} f(x) 
 \label{eq:lemme3init}
\end{equation}

is achieved with $\hat{x}_i=0$.
\end{lemma}

\begin{proof}
We start by proving that $  \sigma(A)\|d\|_2 \geq \left|\left(A^T(A\hat{x}-d)\right)_{i}\right| $. Remark that Lemma \ref{lem:bound2} is valid for problem (\ref{eq:lemme3init}), from which we have 
\begin{align*}
    \sigma(A)\|d\|_2 & \geq  \sigma(A)\|A\hat{x}-d\|_2 \\ 
    &\geq \|A^T\| \|A\hat{x}-d\|_2 \\
    &\geq \|A^T(A\hat{x}-d)\|_2 \\
    &\geq \|A^T(A\hat{x}-d)\|_\infty \\
    &\geq |\left(A^T(A\hat{x}-d)\right)_i| \quad \forall i\in \{1,\dots,N\}
    \end{align*}
    
Then, by choosing, for all $i\in [1..N]$,  $w_{i}>\sigma(A)\|d\|_2$, we are sure that $w_{i}> \left|\left(A^T(A\hat{x}-d)\right)_{i}\right|$.
From the Pshenichnyi-Rockafellar lemma, a necessary and sufficient condition for $\hat{x}$ is a minimizer of $f$ on  $C$ is that 
$$
\textbf{0}\in \partial f(\hat{x})+N_C(\hat{x})
$$
where in our case $C$ is the $\mathbb{R}^N_+$  and $f(x)=\frac{1}{2}\|Ax-d\|^2+<w,|x|>$. We have that $\partial f(x)=\partial(\frac{1}{2}\|Ax-d\|^2) + \partial(<w,|x|>)$ since $f(x)$ is a sum of two convex functions, where the intersection of the domains is non empty (see \cite[Corollary 16.38]{bookConvex}). \\

The optimal condition is therefore
$$
\textbf{0}\in A^T(A\hat{x}-d)+\partial <w,|\hat{x}|> + N_{\mathbb{R}^d_+}(\hat{x})
$$
where 
$$
(\partial <w,|\hat{x}|>)_i \begin{cases}
=w_i \text{ if } \hat{x}_i>0\\
=-w_i \text{ if } \hat{x}_i<0\\
\in [-w_i,w_i] \text{ if } \hat{x}_i=0
\end{cases}
$$
and 
$$
 (N_{\mathbb{R}^d_+}(\hat{x}))_i\begin{cases}
 =0 \text{ if } \hat{x}_i>0\\
 \in ]-\infty,0] \text{ if } \hat{x}_i=0
 \end{cases}
$$
For $\hat{x}_{i}$ we have the following optimal condition
$$
-A^T(A\hat{x}-d)_{i} \begin{cases}
=w_{i} \text{ if } \hat{x}_{i}>0\\
\in [-w_{i},w_{i}] + \,\, ]-\infty,0] \text{ if } \hat{x}_{i}=0
\end{cases}
$$
If $w_i>\sigma(A)\|d\|_2$, then  $|A^T(A\hat{x}-d)_{i}|<w_{i}$ and $\hat{x}_{i}$ cannot be strictly positive, furthermore $\hat{x}_{i}$ cannot be strictly negative since we work in the non-negative space. Therefore  $\hat{x}_{i}=0$.\\
\qed \end{proof}

\begin{lemma}\label{lem:uconst}
Let $(x_\rho,u_\rho)$ be a local minimizer of $G_\rho$ defined in (\ref{eq:Grho}), with $I$ on the constrained form, that is,  defined as in (\ref{eq:a}). Let $G_{x_{\rho}}(u)= \frac{1}{2}\|Ax_{\rho}-d\|^2+ I(u) +\rho(\|x_{\rho}\|_1-<x_{\rho},u>)$.
We denote $O$ as the indexes of the k largest values of  $\{i=1...N,|(x_{\rho})_i|\}$. $Q\triangleq\{i|(x_{\rho})_i>0\}$, and $S\triangleq\{j|(x_{\rho})_j<0\}$. Moreover, we define $D\triangleq O \cap Q$, $L\triangleq O\cap S$ and $W\triangleq\{1,2...,N\}\backslash\{D\cup L\}$. If $\#(D\cup L)=k$, that is, $\|x_\rho\|_0\geq k$, then the minimum of $G_{x_\rho}(u)$ will be reached with $u_{\rho}$ such that 
\begin{equation}\label{eq:hatu}
  (u_{\rho})_i\begin{cases}
=1 \text{ if } i\in D\\
=-1 \text{ if } i\in L\\
=0 \text{ if } i\in W\\
\end{cases}  
\end{equation}
If $\#(D\cup L)<k$, that is, $\|x_\rho\|_0<k$, then 
\begin{equation}\label{eq:hatu}
  (u_{\rho})_i\begin{cases}
=1 \text{ if } i\in D\\
=-1 \text{ if } i\in L\\
\in [-1,1]\text{ if } i\in W\\
\end{cases}  
\end{equation}
such that $\sum_{i\in W} |u_i|\leq k-\#(D\cup L)$.
\end{lemma}
\begin{proof}
We observe that  minimizing $G_{x_\rho}(u)$ can be viewed as a problem of minimizing $-<x_{\rho},u> + \iota_{-1\leq\cdot\leq 1 }(u) + \iota_{\|\cdot\|_1\leq k}(u)$ by using the definition of $I(u)$. 
The results are obvious. 
\qed\end{proof}

\begin{lemma}\label{lem:upenal}
Let $(x_\rho,u_\rho)$ be a local minimizer of $G_\rho$ defined in (\ref{eq:Grho}), with $I$ on the penalized form, that is,  defined as in (\ref{eq:b}). Let $G_{x_{\rho}}(u)= \frac{1}{2}\|Ax_{\rho}-d\|^2+ I(u) +\rho(\|x_{\rho}\|_1-<x_{\rho},u>)$.
The minimum of $G_{x_\rho}(u)$ will be reached with a $u_{\rho}$ such that
\begin{equation}\label{eq:hatu}
  (u_{\rho})_i\begin{cases}
=1 \text{ iff } (x_{\rho})_i\in [\frac{\lambda}{\rho}, + \infty[\\
=0 \text{ iff } (x_{\rho})_i\in \frac{\lambda}{\rho}[-1,1]\\
=-1 \text{ iff } (x_{\rho})_i\in ]-\infty,-\frac{\lambda}{\rho}]\\
\in ]0,1[ \text{ iff } (x_{\rho})_i =\frac{\lambda}{\rho} \\
\in ]-1,0[ \text{ iff } (x_{\rho})_i =-\frac{\lambda}{\rho}
\end{cases}  
\end{equation}
\end{lemma}
\begin{proof}
Proof of the necessary condition: \\
We start by writing the optimal conditions of $G_{x_\rho}(u)$. 
 \begin{align}\label{eq:upenal}
 \textbf{0}&\in -\rho x_{\rho} + N_{-1\leq \cdot \leq 1}(u_{\rho}) + \begin{cases}
\lambda \text{ if } (u_{\rho})_i>0\\
 -\lambda   \text{ if } (u_{\rho})_i<0\\
 [-\lambda ,\lambda ]  \text{ if } (u_{\rho})_i=0
 \end{cases} 
\end{align}

We split the study of (\ref{eq:upenal}) in five cases.
\begin{itemize}
    \item If $(u_{\rho})_i=1$
    $$
    0\in -\rho (x_{\rho})_i + N_{-1\leq \cdot \leq 1}((u_{\rho})_i)+ \lambda \Leftrightarrow  (x_{\rho})_i\in \frac{[0,\infty[+ \lambda}{\rho}
    $$
    Thus, $(u_{\rho})_i=1\Rightarrow (x_{\rho})_i\in [\frac{\lambda}{\rho},+\infty[$
    \item If $0<(u_{\rho})_i<1$
    $$
    0\in -\rho (x_{\rho})_i + N_{-1\leq \cdot \leq 1}((u_{\rho})_i)+ \lambda \Leftrightarrow (x_{\rho})_i= \frac{\lambda}{\rho}
    $$
    Thus $0<(u_{\rho})_i<1 \Rightarrow (x_{\rho})_i=  \frac{\lambda}{\rho}$
    \item If $(u_{\rho})_i=0$
    $$
    0\in -\rho (x_{\rho})_i + N_{-1\leq \cdot \leq 1}((u_{\rho})_i)+ [-\lambda,\lambda]
    \Leftrightarrow
    (x_{\rho})_i\in \frac{\lambda}{\rho}[-1,1]
    $$
    Thus $(u_{\rho})_i=0 \Rightarrow (x_{\rho})_i\in  \frac{\lambda}{\rho}[-1,1]$
    
    \item If $-1<(u_{\rho})_i<0$
    $$
    0\in -\rho (x_{\rho})_i + N_{-1\leq \cdot \leq 1}((u_{\rho})_i)- \lambda
\Leftrightarrow
    (x_{\rho})_i=- \frac{\lambda}{\rho}
    $$
    Thus $-1<(u_{\rho})_i<0\Rightarrow (x_{\rho})_i=  -{\lambda}{\rho}$
   \item If $(u_{\rho})_i=-1$
    $$
    0\in -\rho (x_{\rho})_i + N_{-1\leq \cdot \leq 1}((u_{\rho})_i)- \lambda
    \Leftrightarrow
    (x_{\rho})_i\in \frac{]-\infty,0]- \lambda}{\rho}
    $$
    Thus, $u_{\rho}=-1\Rightarrow (x_{\rho})_i\in ]-\infty, -\frac{\lambda}{\rho}]$
\end{itemize}

Proof of sufficient condition:\\
We can prove that the reverse statement is also true.  We can rewrite $(x_{\rho})_i=\frac{\beta}{\rho}$, for some $\beta \in \mathbb{R}$. We have then from the optimal conditions (\ref{eq:upenal}) that
\begin{align}\label{eq:upenal}
 \textbf{0}&\in -\rho \frac{\beta}{\rho} + N_{-1\leq \cdot \leq 1}(u_{\rho}) + \begin{cases}
\lambda \text{ if } (u_{\rho})_i>0\\
 -\lambda   \text{ if } (u_{\rho})_i<0\\
 [-\lambda ,\lambda ]  \text{ if } (u_{\rho})_i=0
 \end{cases} 
\end{align}

\begin{align}
    0&\in [-\beta + \lambda,+\infty[   \text{ if $(u_{\rho})_i=1$}\label{eq:cond1}\\
    0&\in -\beta + \lambda   \text{ if $0<(u_{\rho})_i<1$}\label{eq:cond2}\\
    0&\in [-\lambda-\beta, \lambda - \beta] \text{ if $(u_{\rho})_i=0$}\label{eq:cond3}\\
    0&\in -\beta - \lambda   \text{ if $-1<(u_{\rho})_i<0$}\label{eq:cond4}\\
    0&\in ]-\infty,-(\beta + \lambda  )]  \text{ if $(u_{\rho})_i=-1$}\label{eq:cond5}\\
\end{align}
Assuming $\beta>\lambda$, then only (\ref{eq:cond1}) is possible. If $\beta=\lambda$, then (\ref{eq:cond1}), (\ref{eq:cond2}) (\ref{eq:cond3}) are possible. If $0\leq \beta<\lambda$, then only (\ref{eq:cond3}) is possible.  If $-\lambda< \beta<0$, then only (\ref{eq:cond3}) is possible.  If $ \beta=-\lambda$, then (\ref{eq:cond3}), (\ref{eq:cond4}) and (\ref{eq:cond5}) are possible. If $\beta<-\lambda$, then only (\ref{eq:cond5}) is possible.

This finishes the proof. 

\qed\end{proof}

\end{document}